\documentclass[11pt]{article}
\usepackage[margin=1in]{geometry} 
\usepackage{amssymb,amsfonts,amsmath,bbm,mathrsfs,stmaryrd}
\usepackage{xcolor}
\usepackage{url}
\usepackage{relsize}

\usepackage[colorlinks,
             linkcolor=black!75!red,
             citecolor=blue,
             pdftitle={BU},
             pdfauthor={x},
             pdfproducer={pdfLaTeX},
             pdfpagemode=None,
             bookmarksopen=true
             bookmarksnumbered=true]{hyperref}
            
\usepackage{braket}
             
\usepackage{tikz}
\usetikzlibrary{arrows,calc,decorations.pathreplacing,decorations.markings,intersections,shapes.geometric,through,fit,shapes.symbols,positioning,decorations.pathmorphing}

\usepackage[amsmath,thmmarks,hyperref]{ntheorem}
\usepackage{cleveref}
\usepackage{comment}

\crefname{section}{Section}{Sections}
\crefformat{section}{#2Section~#1#3} 
\Crefformat{section}{#2Section~#1#3} 

\crefname{subsection}{\S}{\S\S}
\crefformat{subsection}{#2\S#1#3} 
\Crefformat{subsection}{#2\S#1#3} 

\theoremstyle{plain}

\newtheorem{lemma}{Lemma}[section]
\newtheorem{proposition}[lemma]{Proposition}
\newtheorem{corollary}[lemma]{Corollary}
\newtheorem{theorem}[lemma]{Theorem}

\theoremstyle{nonumberplain}

\theoremstyle{plain}
\theorembodyfont{\upshape}
\theoremsymbol{\ensuremath{\blacklozenge}}

\newtheorem{definition}[lemma]{Definition}
\newtheorem{example}[lemma]{Example}
\newtheorem{remark}[lemma]{Remark}

\crefname{definition}{definition}{definitions}
\crefformat{definition}{#2definition~#1#3} 
\Crefformat{definition}{#2Definition~#1#3} 

\crefname{ex}{example}{examples}
\crefformat{example}{#2example~#1#3} 
\Crefformat{example}{#2Example~#1#3} 

\crefname{remark}{remark}{remarks}
\crefformat{remark}{#2remark~#1#3} 
\Crefformat{remark}{#2Remark~#1#3} 

\crefname{convention}{convention}{conventions}
\crefformat{convention}{#2convention~#1#3} 
\Crefformat{convention}{#2Convention~#1#3} 

\crefname{claim}{claim}{claims}
\crefformat{claim}{#2claim~#1#3} 
\Crefformat{claim}{#2Claim~#1#3} 

\crefname{conjecture}{conjecture}{conjectures}
\crefformat{conjecture}{#2conjecture~#1#3} 
\Crefformat{conjecture}{#2Conjecture~#1#3}

\crefname{lemma}{lemma}{lemmas}
\crefformat{lemma}{#2lemma~#1#3} 
\Crefformat{lemma}{#2Lemma~#1#3} 

\crefname{proposition}{proposition}{propositions}
\crefformat{proposition}{#2proposition~#1#3} 
\Crefformat{proposition}{#2Proposition~#1#3} 

\crefname{question}{question}{questions}
\crefformat{question}{#2question~#1#3} 
\Crefformat{question}{#2Question~#1#3}

\crefname{corollary}{corollary}{corollaries}
\crefformat{corollary}{#2corollary~#1#3} 
\Crefformat{corollary}{#2Corollary~#1#3} 

\crefname{theorem}{theorem}{theorems}
\crefformat{theorem}{#2theorem~#1#3} 
\Crefformat{theorem}{#2Theorem~#1#3} 

\crefname{assumption}{assumption}{Assumptions}
\crefformat{assumption}{#2assumption~#1#3} 
\Crefformat{assumption}{#2Assumption~#1#3} 

\crefname{equation}{}{}
\crefformat{equation}{(#2#1#3)} 
\Crefformat{equation}{(#2#1#3)}

\theoremstyle{nonumberplain}
\theoremsymbol{\ensuremath{\blacksquare}}

\newtheorem{proof}{Proof}

\newcommand\bC{{\mathbb C}}

\newcommand\bS{{\mathbb S}}

\newcommand\bZ{{\mathbb Z}}

\newcommand\cF{{\mathcal F}}

\newcommand\cO{{\mathcal O}}

\newcommand\cT{{\mathcal T}}

\newcommand{\qedhere}{\mbox{}\hfill\ensuremath{\blacksquare}}

\def\polhk#1{\setbox0=\hbox{#1}{\ooalign{\hidewidth
    \lower1.5ex\hbox{`}\hidewidth\crcr\unhbox0}}}

\newcommand{\bes}{\begin{equation*}}
\newcommand{\ees}{\end{equation*}}
\newcommand{\be}{\begin{equation}}
\newcommand{\ee}{\end{equation}}

\begin{document}

\title{Gauge freeness for Cuntz-Pimsner algebras}

\author{Alexandru Chirvasitu\footnote{University at Buffalo, \url{achirvas@buffalo.edu}}}

\date{}

\maketitle

\abstract{To every $C^*$ correspondence over a $C^*$-algebra one can associate a Cuntz-Pimsner algebra generalizing crossed product constructions, graph $C^*$-algebras, and a host of other classes of operator algebras. Cuntz-Pimsner algebras come equipped with a ``gauge action'' by the circle group and its finite subgroups.

For unital Cuntz-Pimsner algebras, we derive necessary and sufficient conditions for the gauge actions (by either the circle or its closed subgroups) to be free.}

\vspace{.5cm}

\noindent {\em Key words: Cuntz-Pimsner algebra, $C^*$ correspondence, Hilbert module, free action, strong grading}

\vspace{.5cm}

\noindent{MSC 2010: 46L05; 16S99}

\section*{Introduction}

Among actions of (topological) groups on topological spaces free ones play a special role it is difficult to overstate.

The fact that the quotient space $X\to X/G$ of such an action (by $G$ on $X$) is well behaved makes free actions a good starting point for defining equivariant versions of invariants in algebraic topology. Equivariant cohomology \cite{tu-eq} is an example, as is equivariant K-theory. The latter is typically defined by means of $G$-equivariant vector bundles on $X$, and when the action is free the construction is reducible to ordinary $K$-theory on $X/G$ \cite{seg}.  

The general driving principle behind freeness is that it entails good ``descent'' properties: $G$-equivariant structures on $X$ can be recast as analogous structures on the quotient space $X/G$. This applies for instance to vector bundles: for free actions there is an equivalence
\begin{equation*}
\begin{tikzpicture}[auto,baseline=(current  bounding  box.center)]
  \node[text width=4cm] (1) at (0,0) {vector bundles on X/G};
  \node[text width=4.5cm] (2) at (8,0) {$G$-equivariant vector bundles on $X$};
  \draw[->] (1) .. controls (2,1) and (6,1) .. node{$\simeq$} (2);
  \draw[<-] (1) .. controls (2,-1) and (6,-1) .. node[below] {$\simeq$} (2);
\end{tikzpicture}
\end{equation*}
See e.g. \cite[Proposition 4.2]{seg} and the discussion immediately preceding it. 

For all of these reasons, it is of interest to extend the concept of action freeness to the wider framework of {\it noncommutative} topology, where spaces are recast a possibly noncommutative $C^*$-algebras. This theme, for instance, drives much of \cite{phi-equivk}, where freeness features in the study of operator-algebraic equivariant K-theory.

What we here call `freeness' is referred to in \cite[$\S$7.1]{phi-equivk} as {\it saturation}. For an action of a compact abelian group $G$ on the $C^*$-algebra $A$ this means that
\begin{equation*}
  \overline{A^*_{\tau}A_{\tau}} = A^G,\ \forall \tau\in \widehat{G},
\end{equation*}
where $\widehat{G}$ is the discrete Pontryagin dual of $G$ and $A_{\tau}$ is the $\tau$-eigenspace of the action (see \Cref{se.prel} below for further details). We also refer the reader to \cite{phi-free} for more discussion of this non-commutative analogue to classical freeness for group actions on operator algebras. The concept of freeness is extensible even further, to the setting of actions of {\it quantum} groups on $C^*$-algebras; \cite{ell,freeness} are good references for this aspect of the story.

Whether classical or quantum, the interest in freeness ultimately stems from the same good descent properties it brings about alluded to above. For actions of an ordinary (as opposed to quantum) compact group $G$ on a $C^*$-algebra $A$, for instance, saturation is initially {\it defined } in \cite[7.1]{phi-equivk} as the requirement that a certain bimodule implement a Morita equivalence between the fixed-point algebra $A^G$ (noncommutative analogue to $X/G$) and the crossed product $A\rtimes G$ (see e.g. Definition 7.1.4 in loc. cit.; the notion was initially due to M. Rieffel).

The category of finitely generated projective $A\rtimes G$-modules is the noncommutative incarnation of equivariant bundles, so the story is indeed parallel to its classical counterpart. Although quantum group actions will not feature below, the same analogies persist when the acting group is quantum as well. This follows, for instance, from \cite[Definition 0.2 and the equivalence (1) $\Leftrightarrow$ (2) of Theorem 0.4]{freeness} combined with \cite[the equivalence (1) $\Leftrightarrow$ (2) of Theorem I]{sch}.  

The present paper is concerned with actions by the circle group and its finite subgroups $\bZ/k$ on the $C^*$-algebras introduced in \cite{kat-init}, modifying the initial construction of \cite{pim}. The initial data of the construction is a {\it $C^*$-correspondence} consisting of a $C^*$-algebra $A$, a Hilbert $A$-module $E$, and a morphism $\varphi:A\to B(E)$ into the $C^*$-algebra of adjointable operators on $E$.

This makes $E$ into an $A$-bimodule and allows the construction of the {\it Fock space}
\begin{equation*}
\cF(E)=A\oplus E \oplus E^{\otimes 2}\oplus \cdots
\end{equation*}
attached to the correspondence. One then construct the {\it Toeplitz algebra} $\cT_E$. This is the $C^*$-subalgebra of $B(\cF(E))$ generated by creation operators
\begin{equation*}
  T_e: E^{\otimes n}\to E^{\otimes (n+1)}\text{ defined by }
  \begin{cases}
    a\mapsto ea & \text{ if }n=0\\
    e_1\otimes\cdots \otimes e_n \mapsto e\otimes e_1\otimes\cdots\otimes e_n & \text{otherwise}
  \end{cases}
\end{equation*}
for $e\in E$ and $A$, where the latter acts on the left on each $E^{\otimes n}$ through $\varphi$. Finally, the Cuntz-Pimsner algebra $\cO_E$ is a certain quotient of $\cT_E$ (see \Cref{subse.cp} below for details). 

The construction generalizes crossed products, graph algebras and other popular classes of $C^*$-algebras (see e.g. \cite{kat-init} for further discussion). The algebras $\cT_E$ and $\cO_E$ come equipped with a ``gauge action'' by the circle group $bS^1\subset \bC^{\times}$ whereby $z\in \bS^1$ scales $E$ by $z$; this then also restricts to actions by the finite subgroups of $\bS^1$.

It is these gauge actions whose freeness we are interested in throughout. The paper is organized as follows:

\Cref{se.prel} contains some prefatory material setting the stage for the sequel.

In \Cref{se.fr} we prove the main results: \Cref{th.cp-iff,th.fin} characterize those unital $C^*$ correspondences $(A,E,\varphi)$ for which the gauge action by $\bS^1$ (respectively $\bZ/k\subset \bS^1$ for positive integers $k\ge 2$) are free. 

The same last section also contains a discussion, in \Cref{subse.prior}, of the relation between the main result of the present paper and earlier work in the literature on the freeness of gauge actions on graph $C^*$-algebras.

After reviewing briefly how the latter can be recovered as Cuntz-Pimsner algebras, we recall the characterization of graphs for which the various gauge actions are free due to \cite{chr} and re-derive these results as consequences of \Cref{th.cp-iff,th.fin} for graphs with finitely many vertices.

\section{Preliminaries}\label{se.prel}

\subsection{Cuntz-Pimsner algebras}\label{subse.cp}

As noted above, the main objects of study are the algebras attached to a Hilbert correspondence $E$ over a (unital, in our case) $C^*$-algebra $A$. introduced in \cite{kat-init} and further studied in \cite{kat-alg,kat-ideal}. The construction is a modification of that in \cite{pim}, and the two coincide when the left action of $A$ on $E$ is faithful.

We now briefly sketch the construction. First, recall:

\begin{definition}
  Let $A$ be a $C^*$-algebra. A {\it correspondence} over $A$ is a (right) Hilbert module $E$ over $A$ together with a morphism $\varphi:A\to B(E)$ into the algebra of adjointable operators on $E$.

  The correspondence is {\it faithful} if $\varphi$ is one-to-one and {\it full} if the closed span of the elements
  \begin{equation*}
    A\ni \braket{x|y},\ x,y\in E
  \end{equation*}
  (which in general is an ideal of $A$) is all of $A$.

  The correspondence is {\it finitely generated} (or FG for short) if $E$ is so as a right $A$-module.
\end{definition}

\begin{remark}\label{re.pfg}
  Finite generation is equivalent to the ideal $K(E)\subseteq B(E)$ of compact operators being unital, i.e. all of $B(E)$; it is also equivalent to $E$ being projective and finitely generated as a right $A$-module (see \cite[Theorem 15.4.2]{wo} and surrounding discussion). 
\end{remark}

To a $C^*$ correspondence $(A,E,\varphi)$ \cite{kat-init} associates a $C^*$-algebra denoted here somewhat abusively by $\cO_E$. While a bit of a misnomer, we will refer to these as {\it Cuntz-Pimsner algebras}, which phrase we will occasionally abbreviate to CP. The ambiguity stems from the fact that the construction differs somewhat from that of \cite{pim}, but the present paragraph will serve as sufficient warning to the reader.  

\begin{remark}
  For us, the underlying $C^*$-algebra $A$ of a correspondence is always unital, as is the left action via the morphism $\varphi:A\to B(E)$. We occasionally remind the reader of this convention by referring to the correspondence itself as {\it unital}. 
\end{remark}

The construction of $\cO_E$ is recalled in \cite[$\S$3]{kat-alg}. One starts by defining the {\it Pimsner-Toeplitz} algebra $\cT_E$ attached to $E$ as being generated by the creation and annihilation operators $e$ and $e^*$, $e\in E$ on the {\it Fock space}
\begin{equation*}
  \mathcal{F}(E):= A\oplus E\oplus E^{\otimes 2}\oplus\cdots. 
\end{equation*}
The tensor powers $E^{\otimes n}$ are Hilbert correspondences using the two actions of $A$ on $E$, and the direct sum is one of Hilbert correspondences (see also \cite[$\S$4]{kat-alg}).

One can then define (injective, it turns out) morphisms $\psi_n: K(E^{\otimes n})\to \cT_E$ by
\begin{equation*}
  K(E^{\otimes n})\ni \ket{e_1\otimes\cdots\otimes e_n}\bra{f_1\otimes\cdots\otimes f_n} \mapsto e_1\cdots e_n f^*_1\cdots f^*_n \in \cT_E. 
\end{equation*}
Define the ideal $J_E$ of $A$ by
\begin{equation*}
  J_E= \varphi^{-1}(K(E))\cap \ker(\varphi)^{\perp} = \{a\in A\ |\ a\ker(\varphi) = \ker(\varphi)a = 0\text{ and }\varphi(a)\in K(E)\},
\end{equation*}
where in general, for an ideal $I\subseteq A$, its orthogonal complement $I^\perp$ is the ideal of elements that annihilate $I$ by multiplication (see e.g. \cite[Definition 3.2]{kat-alg}). 

Finally, we have

\begin{definition}\label{def.cp-def}
  The {\it Cuntz-Pimsner algebra} $\cO_E$ associated to the $C^*$ correspondence $(A,E,\varphi)$ is the quotient of the Pimsner-Toeplitz algebra $\cT_E$ by the ideal generated by
  \begin{equation*}
    a - \psi_1(\varphi(a))
  \end{equation*}
  as $a$ ranges over $J_E$. 
\end{definition}

The surjection $\cT_E\to \cO_E$ induces (again, one-to-one) morphisms  $K(E^{\otimes n})\to \cO_E$; by a slight notational abuse we once more denote them by $\psi_n$ respectively. We rely on context to distinguish between these and their identically-named counterparts with codomain $\cT_E$.  

For both $\cT_E$ and $\cO_E$ the span of the elements
\begin{equation*}
  a\in A,\quad e_1\cdots e_n f^*_1\cdots f^*_m,\ e_i,f_j\in E,\ m,n\in \bZ_{\ge 0} \text{ and }m+n>0
\end{equation*}
is a dense $*$-subalgebra.

\begin{definition}\label{def.gg}
  The {\it gauge action} on $\mathcal{T}_E$ or $\mathcal{O}_E$ is the action of $\bS^1$ on $\cO_E$ and $\cT_E$ whereby $z\in \bS^1\subset \bC$ scales $e\in E$ by $z$.

  We also refer to the restrictions to $\bZ/k\subset \bS^1$ as gauge actions, and when having to distinguish call the $\bS^1$-action {\it full}.  
\end{definition}

\subsection{Graph algebras}\label{subse.gr-cp}

Note that graph $C^*$-algebras (e.g. \cite[Definition 1 and surrounding discussion]{flr}) can be recovered as Cuntz-Pimsner algebras $\cO_E$. Our reference for this will be \cite[$\S$3.4]{kat-init}. 

For a directed graph $\Gamma$ with vertex set $\Gamma^0$, edge set $\Gamma^1$ and source and range maps
\begin{equation*}
  s,r:\Gamma^1\to \Gamma^0
\end{equation*}
we have

\begin{definition}\label{def.gr-cast}
  The {\it graph $C^*$-algebra} $C^*(\Gamma)$ is the universal $C^*$-algebra generated by vertex projections $p_v$, $v\in \Gamma^0$ and partial isometries $s_\gamma$, $\gamma\in \Gamma^1$ subject to the following relations
  \begin{itemize}
  \item $s^*_{\gamma}s_{\gamma}=p_{r(\gamma)}$;
  \item for every vertex $v$, the ranges of $s_{\gamma}$ for $s(\gamma)=v$ are mutually orthogonal and dominated by $p_{s(\gamma)}$;
  \item if the set $s^{-1}(v)$ is non-empty and finite then
    \begin{equation*}
      \sum_{\gamma\in s^{-1}(v)} s_{\gamma}s^*_{\gamma}. 
    \end{equation*}
  \end{itemize}
\end{definition}

Now set $A=C_0(\Gamma^0)$ ($C^*$-algebra of functions vanishing at infinity on the discrete space $\Gamma^0$) and
\begin{equation*}
  E=\left\{f:E^1\to \bC\ |\ \Gamma^0\ni v\mapsto \sum_{r(\gamma)=v}|f(\gamma)|^2\in C_0(\Gamma^0)\right\}. 
\end{equation*}
The right action
\begin{equation*}
  E\times A\ni (f,\xi)\mapsto f\triangleleft \xi\in E
\end{equation*}
giving the Hilbert module structure is defined by
\begin{equation*}
  (f\triangleleft\xi)(\gamma) = f(\gamma)\xi(r(\gamma)). 
\end{equation*}
Similarly, the left action $\triangleright$ defining $\varphi:A\to B(E)$ is given by
\begin{equation*}
  (\xi\triangleright f)(\gamma) = \xi(s(\gamma))f(\gamma).  
\end{equation*}
The $A$-valued inner product on $E$ is
\begin{equation*}
  \braket{f|g} = \left(\Gamma^0\ni v\mapsto \sum_{r(\gamma)=v}\overline{f(\gamma)}g(\gamma)\in \bC\right) \in C_0(\Gamma^0).
\end{equation*}

This defines a correspondence $(A,E,\varphi)$ over $A$ which we sometimes denote somewhat abusively by $E(\Gamma)$. With all of this in place, $C^*(\Gamma)$ is isomorphic to $\cO_{E(\Gamma)}$ (see \cite[$\S$3.4]{kat-init}). This relationship between graph and CP algebras will resurface later in \Cref{subse.prior}, when we address connections to existing literature.

$C^*(\Gamma)$ is unital if and only if $\Gamma^0$ is finite (e.g. \cite[Proposition 1.4]{kpr}), so we will be concerned primarily with graphs with finitely many vertices. 

\subsection{Free actions}\label{subse.fr}

Consider an action of a compact group $G$ on a unital $C^*$-algebra $A$, cast as a coaction
\begin{equation*}
  \rho: A\to A\otimes C(G). 
\end{equation*}

Recall e.g. \cite[Definition 2.4]{ell} or \cite[Definition 0.1]{freeness}:

\begin{definition}\label{def.free}
  The action is {\it free} if the map defined by
  \begin{equation*}
    A\otimes A\ni a\otimes b\mapsto (a\otimes 1)\rho(b)\in A\otimes C(G)
  \end{equation*}
has norm-dense range.   
\end{definition}

As confirmed by \cite[Theorem 2.9]{ell}, this recovers the classical notion of freeness when $A=C(X)$ for a compact Hausdorff space $X$. \cite{ell,freeness} both work in the context of compact {\it quantum} groups acting on $A$; we do not need the more elaborate framework of quantum groups here, given that we are concerned with actions by $\bS^1$ an its finite subgroups only.

For compact abelian groups freeness is also equivalent to the notion of {\it saturation} of \cite[Definition 5.2]{phi-free} (see also \cite[Definition 5.9 and Theorem 5.10]{phi-free}). We review this briefly. 

\begin{definition}\label{def.sat}
  Let $G$ be a compact abelian group acting on a $C^*$-algebra $A$. For every character $\tau$ of $G$ denote by $A_\tau$ the $\tau$-eigenspace of the action:
\begin{equation*}
  A_\tau=\{a\in A\ |\ ga = \tau(g)a,\ \forall g\in G\}. 
\end{equation*}
The action is {\it saturated} if $A^*_{\tau}A_{\tau}$ is dense in the fixed-point subalgebra $A_1=A^G$ for all $\tau\in \widehat{G}$. 
\end{definition}
Freeness is equivalent to saturation, which will be our preferred incarnation of the former throughout the paper.

\section{Free gauge actions}\label{se.fr}

The current sections contains the main results of the paper, characterizing those $C^*$ correspondences $(A,E,\varphi)$ for which the various gauge actions (full, by $\bS^1$, or finite, by its various subgroups $\bZ/k$) are free.

\subsection{Full action}\label{subse.full}

The analogue of \cite[Proposition 2]{s-w03} for CP algebras is

\begin{proposition}\label{pr.cp-suff}
Let $(A,E,\varphi)$ be a faithful, full, FG $C^*$ correspondence. Then, the gauge action of $\mathbb{S}^1$ on $\cO_E$ is free.   
\end{proposition}
\begin{proof}
  Write $B=\cO_E$. We have to show that $1\in B$ belongs to the closed linear span of $B_d B^*_d$ for $d=\pm 1$.

  {\bf ($d=1$)} The elements of $E\subset B$ have degree one, so it will suffice to argue that $1\in A\subset B$ can be approximated by sums of the form
  \begin{equation}\label{eq:eifi}
    \sum_i e_i f^*_i\in B,\ e_i, f_i\in E. 
  \end{equation}
The fact that $E$ is FG and faithful means that $1\in B(E)$ is compact, i.e. in the closed linear span of the elementary compact operators $\ket{e}\bra{f}$ for $e,f\in E$. These, however, get identified respectively with $ef^*$ through the canonical map $K(E)\to B$ resulting from the $C^*$-correspondence structure. Since $(A,E,\varphi)$ is faithful $1\in A$ is identified in $B$ with its counterpart in $K(E)=B(E)$, and hence sums of the form (\ref{eq:eifi}) do indeed approximate the identity in $B$ arbitrarily well. 
  
{\bf ($d=-1$)} For $e,f\in E$ we have the identification
\begin{equation*}
  B\ni f^*e = \braket{f|e}\in A\subset B,
\end{equation*}
so our assumption that the closed span of such elements is all of $A$ implies that the latter is contained in the closure of $B^*_1 B_1$.
\end{proof}

We also have a partial converse to \Cref{pr.cp-suff}:

\begin{proposition}\label{pr.cp-nec}
  Suppose the gauge $\mathbb{S}^1$-action on $\cO_E$ is free. Then, the $C^*$ correspondence is faithful and FG. 
\end{proposition}
\begin{proof}
  We prove the two claims separately.

  {\bf Faithfulness.} Suppose the structure map $\varphi:A\to B(E)$ inducing the left action of $A$ on $E$ is {\it not} one-to-one and let $0\ne a\in \ker \varphi$. Then, for each $e\in E\subset B-\cO_E$ we have $ae=0$ in $B$. We thus have
  \begin{equation*}
    B\ni ae_1\cdots e_n f^*_1\cdots f^*_{n-1}=0,\ \forall e_i,f_j\in E,\ \forall n\ge 1. 
  \end{equation*}
  Since the degree-one component $B_1$ is the closed span of such words consisting of $e$s and $f$s we have $aB_1=0$, contradicting the fact that the action is free and hence $1\in\overline{B_1 B^*_1}$. 
  
  {\bf FG.} The freeness assumption implies in particular that $1\in \overline{B_1 B^*_1}$. Using the fact that $B_1$ is the closed span of elements of the form
  \begin{equation*}
    e_1\cdots e_n f^*_1\cdots f^*_{n-1},\ \forall e_i,f_j\in E,\ n\ge 1,
  \end{equation*}
we conclude that $1$ is in the closed linear span of elements 
\begin{equation}\label{eq:enfn}
  e_1\cdots e_n f^*_1\cdots f^*_{n},\ \forall e_i,f_j\in E,\ n\ge 1.
\end{equation}
As observed before, $ef^*$ for $e,f\in E$ is the image of the compact operator $\ket{e}\bra{f}\in K(E)$ through the canonical map $K(E)\to \cO_E$. The same remark applies to larger $n$: the element (\ref{eq:enfn}) is the image of the compact operator
\begin{equation*}
  \ket{e_1\otimes\cdots \otimes e_n}\bra{f_1\otimes \cdots\otimes f_n} \in K(E^{\otimes n})\subset B. 
\end{equation*}
For $m\ge 0$, the $C^*$-subalgebra of $\cO_E$ generated by the images of $K(E^{\otimes n})$ with $n\ge m$ is denoted by $\mathcal{B}_{[m,\infty]}$ in \cite[$\S$5]{kat-alg}. The bottom exact sequence of \cite[Proposition 5.16]{kat-alg} then says that we have an isomorphism
\begin{equation}\label{eq:kat-cong}
  \mathcal{B}_{[0,\infty]}/\mathcal{B}_{[1,\infty]}\cong A/\varphi^{-1}(K(E)). 
\end{equation}
On the other hand though, our assumption that
\begin{equation*}
  A\ni 1\in \overline{\mathrm{span}}\{e_1\cdots e_n f^*_1\cdots f^*_{n}\ |\ e_i,f_j\in E,\ n\ge 1\} = \mathcal{B}_{[1,\infty]}
\end{equation*}
means precisely that the left hand side of \Cref{eq:kat-cong} is trivial. It follows that $K(E)$ contains $1\in A$ and hence is unital, meaning that indeed $E$ is algebraically finitely generated. 
\end{proof}

We note that the two necessary conditions from \Cref{pr.cp-nec} are in general not sufficient:

\begin{example}\label{ex.not-art}
  Let $A$ be the product $\ell^{\infty}(\mathbb{N})$ (in the category of $C^*$-algebras) of countably many copies of $\mathbb{C}$, where $\mathbb{N}$ is the set of all non-negative integers. We set $E=\chi A$, where $\chi\in A$ is the characteristic function of the subset $\mathbb{N}_{\ge 1}\subset \mathbb{N}$ of {\it positive} integers.

  Finally, in order to complete the description of $E$ as a $C^*$ correspondence over $A$, let
  \begin{equation*}
    \varphi:A\to B(E) = \chi A\chi = \chi A
  \end{equation*}
  be the isomorphism induced by the shift: $f\in \ell^{\infty}$ is mapped to the function taking the value $0$ at $0\in \mathbb{N}$ and $f(n-1)$ at $n\ge 1$.

  The FG property is obvious ($eA$ is a cyclic right $A$-module), as is the faithfulness, given that we have defined $\varphi$ to be an isomorphism. \Cref{le.not-art} below, on the other hand, argues that the resulting full gauge action on $B=\cO_E$ is not free
\end{example}

\begin{lemma}\label{le.not-art}
  For the $C^*$ correspondence $(A,E,\varphi)$ from \Cref{ex.not-art} the element $1\in B=\cO_E$ is not in the closure of $B^*_1 B_1 = B_{-1}B^*_{-1}$. 
\end{lemma}
\begin{proof}
$B_{-1}=B^*_1$ is the closure of the span of elements of the form 
\begin{equation}\label{eq:efs}
  e_1\cdots e_n f^*_1\cdots f^*_{n+1}=0,\ \forall e_i,f_j\in E,\ n\ge 0. 
\end{equation}
Moreover, faithfulness together with the FG property ensure that $1$ is arbitrarily approximable by sums of the form $\sum_{e,f\in E}ef^*$, so we can work with elements (\ref{eq:efs}) for a fixed but arbitrary $n$. In conclusion, it suffices to show that the distance from $1$ to any linear combination of elements of the form
\begin{equation}\label{eq:efpq}
  e_1\cdots e_n f^*_1\cdots f^*_{n+1} p_1\cdots p_{n+1} q^*_{1}\cdots q^*_{n}
\end{equation}
is $\ge 1$.

Denote the characteristic function of $\mathbb{N}_{\ge j}\subset \mathbb{N}$ by $\chi_j$, so that $\chi_0=1$, $\chi_1$ is the element $\chi$ from \Cref{ex.not-art}, etc. Note that the elements
\begin{equation*}
  f^*_1\cdots f^*_{n+1} p_1\cdots p_{n+1}\in A\subset B
\end{equation*}
belong to the ideal $\chi_{n+1}A$.

On the other hand, it is easy to see that $E^{\otimes n}$ is $\chi_n A$ as a Hilbert module, with the left hand action given by the isomorphism $A\to B(\chi_n E) = \chi_n A\chi_n$ induced by $n$-fold shift (this is analogous to the case $n=1$ described in \Cref{ex.not-art}).

Regarding products $e_1\cdots e_n$ as elements of $E^{\otimes n}$, the above remarks show that (\ref{eq:efpq}) is an element of $B(E^{\otimes n})\subset B$ taking values in the proper subspace
\begin{equation*}
  \chi_{n+1}E = (\chi_n E)\chi_{n+1}\subset \chi_n E. 
\end{equation*}
This shows that indeed a linear combination of such elements is $\ge 1$ away from the unit $1\in B(E^{\otimes n})\subset B$.
\end{proof}

\Cref{ex.not-art} indicates what can go wrong when attempting to reverse the implication in \Cref{pr.cp-nec}: the issue there is that for $n\ge 0$ we have
\begin{equation*}
  e^*_1\cdots e^*_n f_1\cdots f_n \in \chi_n A\subset A\subset \cO_E,
\end{equation*}
forming a strictly descending chain of ideals. As soon as we rule this out, however, freeness of the gauge action is restored. To phrase the result we need to introduce the following notion.

\begin{definition}\label{def.artin}
  Let $(A,E,\varphi)$ be a $C^*$ correspondence, and for non-negative integers $n\ge 0$ denote by $I_n$ the ideal of $A$ defined recursively as follows:
  \begin{itemize}
  \item $I_0=A$;
  \item $I_1$ is the closed span of $\braket{e|f}$ for $e,f\in E$;
  \item $I_{n+1}$ is the closed span of $\braket{e|\varphi(x)f}$ for $e,f\in E$ and $x\in I_n$.     
  \end{itemize}
  Note that we have $I_0\supseteq I_1\supseteq\cdots$.
  
  The $C^*$ correspondence is {\it artinian} if the chain $(I_n)_{n\ge 0}$ eventually stabilizes. 
\end{definition}

\begin{remark}\label{re.artin-meaning}
  The artinian property is here named by analogy with the algebraic notion, where it means that an algebra has no strictly descending infinite chains of ideals.

  For graph algebras the property is analogous to condition (Y) of \cite[Definition 4.1]{chr}.   
\end{remark}

\begin{theorem}\label{th.cp-iff}
  Let $(A,E,\varphi)$ be a unital $C^*$-correspondence. The gauge $\mathbb{S}^1$-action on the corresponding Cuntz-Pimsner algebra $B=\cO_E$ is free if and only if the correspondence is
  \begin{itemize}
  \item faithful;
  \item FG;
  \item artinian.    
  \end{itemize}
\end{theorem}
\begin{proof}
  {\bf ($\Leftarrow$)} Suppose the correspondence has the three properties in the statement. Faithfulness and the FG property imply that in $B$ we have
  \begin{equation*}
    1\in \overline{B_1 B^*_1},
  \end{equation*}
  as in the proof of \Cref{pr.cp-suff}. We have to further prove that $1\in \overline{B_{-1}B^*_{-1}}$. This entails showing, as described in the proof of \Cref{le.not-art}, that the unit $1\in B$ is in the closed span of elements of the form \Cref{eq:efpq} for some $n\ge 0$.

  In $B$, the middle segments $$f^*_1\cdots f^*_{n+1} p_1\cdots p_{n+1}$$ of \Cref{eq:efpq} respectively span (a dense subset of) the ideal $I_{n+1}$, so for sufficiently large $n$ linear combinations of such elements approximate elements of the form
  \begin{equation*}
    f^*_1\cdots f^*_{n} p_1\cdots p_{n}
  \end{equation*}
  arbitrarily well. This means that for large $n$ 
  \begin{equation*}
    e_1\cdots e_n f^*_1\cdots f^*_{n} p_1\cdots p_{n} q^*_{1}\cdots q^*_{n}
  \end{equation*}
  for $e_i\in E$, etc. is approximable by linear combinations of elements of the form \Cref{eq:efpq}. In other words, we have
  \begin{equation*}
    1\in \overline{B_0 B^*_0}\subseteq \overline{B_{-1} B^*_{-1}},
  \end{equation*}
as desired. 
  
{\bf ($\Rightarrow$)} We already know from \Cref{pr.cp-nec} that freeness entails faithfulness and the FG property, so it remains to show that it also implies artinianness. Freeness implies that $1\in \overline{B_{-1}B^*_{-1}}$, and hence for some $n$ a linear combination $t\in B$ of elements of the form \Cref{eq:efpq} is within some small $\varepsilon>0$ of the unit $1\in B$.

Such a $t$ is an operator on
\begin{equation*}
  E^{\otimes n} = \overline{\mathrm{span}}\{e_1\cdots e_n\}\subset B
\end{equation*}
taking values in $E^{\otimes n} I_{n+1}$. The previous paragraph then implies that
\begin{equation*}
  \mathrm{dist}(x,\overline{E^{\otimes n}I_{n+1}})\le \|tx-x\| < \varepsilon\|x\|
\end{equation*}
for all $x\in E^{\otimes n}$, meaning (e.g. by Hahn-Banach) that the Hilbert submodule
\begin{equation}\label{eq:hm-incl}
  \overline{E^{\otimes n}I_{n+1}}\subseteq E^{\otimes n}
\end{equation}
is equal to $E^{\otimes n}$. 

To conclude, note that on the one hand left-multiplying the left hand side of (what we now know to be the equality) \Cref{eq:hm-incl} by
\begin{equation*}
  e_1^*\cdots e_n^*,\ e_i\in E
\end{equation*}
produces elements in $I_{n+1}$, whereas applying the same maps to the right hand side produces elements spanning a dense subspace of $I_n$. The claim that $I_{n+1}=I_n$ follows.
\end{proof}

\subsection{Finite actions}\label{subse.fin}

Next, we consider the restricted gauge actions of $\mathbb{Z}/k$, $k\ge 2$ on Cuntz-Pimsner algebras $\cO_E$, for the same purpose of determining precisely when they are free. We retain the shorthand notation $B=\cO_E$, and recall from \Cref{subse.cp} the definition of the ideal $J_E\subseteq A$: 
\begin{equation*}
  J_E = \varphi^{-1}(K(E))\cap \ker(\varphi)^\perp,
\end{equation*}
where the {\it orthogonal complement} $I^\perp$ of an ideal $I\subseteq A$ is by definition
\begin{equation*}
  I^\perp=\{a\in A\ |\ ax=0,\ \forall x\in I\}. 
\end{equation*}

Recall also the ideals $I_n$, $n\ge 0$ of $A$ introduced in \Cref{def.artin}. The main result of the present subsection is

\begin{theorem}\label{th.fin}
  Let $k\ge 2$ be a positive integer and $(A,E,\varphi)$ a unital $C^*$ correspondence. The gauge $\mathbb{Z}/k$-action on $B=\cO_E$ is free if and only if
  \begin{equation}\label{eq:k-1}
    \overline{J_E+I_{k-1}} = A. 
  \end{equation}
\end{theorem}
\begin{proof}
  We prove the two claimed implications separately.
  
  {\bf ($\Leftarrow$)} We have to argue that under the assumption that \Cref{eq:k-1} holds we have $1\in \overline{B_1 B^*_1}$ for a generator $1\in \mathbb{Z}/k$. Essentially by definition, we have
  \begin{equation*}
    I_{k-1} = \overline{\mathrm{span}}\{e^*_1\cdots e^*_{k-1} f_1\cdots f_{k-1}\ | e_i,f_j\in E\} \subset A\subset B.
  \end{equation*}
  On the other hand $K(E)\subset B$ is the closed span of $ef^*$ for $e,f\in E$, and by the definition of the Cuntz-Pimsner algebra $B$ the elements $x\in J_E\subset A\subset B$ are identified respectively with $\varphi(x)\in K(E)\subset B$. It follows that such elements are approximable sums
  \begin{equation*}
    \sum_{e,f\in E} ef^*.
  \end{equation*}
  All in all, the hypothesis ensures that we can find an invertible element of $B$ of the form
  \begin{equation*}
    \sum ef^* + \sum e^*_1\cdots e^*_{k-1} f_1\cdots f_{k-1}
  \end{equation*}
  for $e$s and $f$s in $E$. Since all terms belong to $B_1 B^*_1$, this concludes the proof of the present implication.

  {\bf ($\Rightarrow$)} This time we assume that $1\in \overline{B_1 B^*_1}$. The monomials (in $e$, $f^*$ for $e,f\in E$) belonging to $B_1$ are qualitatively of two types:
  \begin{itemize}
  \item $e_1\cdots e_t f^*_1\cdots f^*_s$ for $t-s\cong 1(\mathrm{mod}\ k)$;
  \item $e^*_1\cdots e^*_s$ for $s\cong -1(\mathrm{mod}\ k)$   
  \end{itemize}
  $1$ is approximable by a finite sum $\sigma$ of $xy^*$ for linear combinations $x$ and $y$ of terms of the above two types. Discarding those terms of $\sigma$ that lie outside the degree-$0$ component of $B$ with respect with the full gauge (i.e. $\mathbb{S}^1$-)action, we obtain an approximation of $1$ by
  \begin{equation*}
    \sum e_1 \cdots e_t f^*_1\cdots f^*_t + \sum g^*_1\cdots g^*_s h_1\cdots h_s 
  \end{equation*}
  for $e$s, $f$s, $g$s and $h$s in $E$, various $t\ge 1$ and $0<s\cong -1(\mathrm{mod}\ k)$ (so in particular $s\ge k-1$). In the notation of \cite[$\S$5]{kat-alg} and recalled in the proof of \Cref{pr.cp-nec} above, this means that
  \begin{equation*}
     \overline{\mathcal{B}_{[1,\infty]}+I_{k-1}} = \mathcal{B}_{[0,\infty]}. 
  \end{equation*}
  Once more using the result (of \cite[Proposition 5.16]{kat-alg}) that the canonical map
  \begin{equation*}
    A/J_E\to \mathcal{B}_{[0,\infty]}/\mathcal{B}_{[1,\infty]}
  \end{equation*}
  is an isomorphism, we conclude that the image of $J_E$ through $A\to A/J_E$ is full. This is the desired conclusion, finishing the proof.
\end{proof}

\begin{remark}
  Once more, this result shadows the analogous one for graph algebras (which it generalizes). Compare the condition in the statement, referencing the ideal $I_{k-1}$, with the necessary and sufficient condition for freeness in \cite[Proposition 4.3]{chr} for graph algebras: in the context of row-finite graphs it requires that all sinks receive length-$(k-1)$ paths. 
\end{remark}

\subsection{Prior work}\label{subse.prior}

Freeness for gauge actions has been considered before by several authors in the context of graph algebras. In this subsection we review how the above results recover those in the literature (at least for graphs with finitely many vertices) via the realization of graph $C^*$-algebras as Cuntz-Pimsner sketched in \Cref{subse.gr-cp}.

\cite[Proposition 2]{s-w03} provides a sufficient condition for freeness for the full gauge action: the graph $\Gamma$ should meet the requirements that
\begin{itemize}
\item every vertex emits only finitely many edges (we then say that the graph is {\it row-finite} because its adjacency matrix has rows with finitely many non-zero entries);
\item every vertex emits at least one edge (i.e. there are no {\it sinks});
\item every vertex receives at least one edge (i.e. there are no {\it sources}). 
\end{itemize}

The result is not optimal, in the sense that the implication is not reversible. A complete characterization of graphs with free gauge actions (both full and finite) is achieved in \cite[Theorem 4.2 and Proposition 4.3]{chr}. The results therein are phrased not for graph $C^*$-algebras but rather for the {\it Leavitt path algebras} $L(\Gamma)$ attached to the graph $\Gamma$.

By definition, $L(\Gamma)\subset C^*(\Gamma)$ is the dense $*$-subalgebra generated by all vertex projections $p_v$, $v\in \Gamma^0$ and all partial isometries $s_\gamma$, $\gamma\in \Gamma^1$. In that purely algebraic context, the counterpart of freeness is {\it strong gradedness}.

The gauge actions by $G=\bS^1$ or the various finite subgroups $G=\bZ/k$ induces a grading of $L(\Gamma)$ by the character group $\widehat{G}$. The latter group is isomorphic to $\bZ$ for the full action or $\bZ/k$ for the finite gauge actions, and the grading assigns the partial isometries $s_{\gamma}$ degree one. 

We then have
\begin{definition}\label{def.strng-grd}
  A grading of $L=L(\Gamma)$ by a group $\widehat{G}$ is {\it strong} if $L_\tau L_\eta=L_{\tau\eta}$. 
\end{definition}
The condition is easily seen to be stronger than that of saturation in \Cref{def.sat}. \Cref{pr.dual-pb} below says that the two are in fact equivalent, so that the two versions of the problem (purely algebraic via $L(\Gamma)$ vs. $C^*$-algebraic via $C^*(\Gamma)$) coincide.

We will need some preparation.

\begin{definition}\label{def.loc-unit}
  An associative algebra $A$ is {\it locally unital} if for every finite subset $F\subseteq A$ there is some $a\in A$ such that
  \begin{equation*}
    ax = xa = x,\ \forall x\in F. 
  \end{equation*}
When $A$ happens to be graded by some group $\widehat{G}$ we will usually assume further that the local units $a$ as above belong to $A_1$. 
\end{definition}

\begin{lemma}\label{le.loc-unit-prel}
Let $L$ be a locally unital $\widehat{G}$-graded algebra. The grading is strong if and only if $L_{\tau}L_{\tau^{-1}}=L_1$ for $\tau$ ranging over a set that generates $\widehat{G}$ as a monoid. 
\end{lemma}
\begin{proof}
One implication (strong grading $\Rightarrow$ the condition in the statement) is immediate, so we focus on the other one. Suppose then that $L_{\tau}L_{\tau^{-1}}=L_1$ for $\tau$ ranging over a set $F$ generating $\widehat{G}$ as a monoid.

{\bf Claim: $L_{\tau}L_{\tau^{-1}}=L_1$ for {\it all} $\tau$.} Indeed, every element $\tau$ is expressible as a product $\tau_1\cdots\tau_n$ for $\tau_i\in F$, and
\begin{equation*}
  L_{\tau}L_{\tau^{-1}}\supseteq \left(L_{\tau_1}\cdots \left(L_{\tau_n} L_{\tau_n^{-1}}\right)\cdots L_{\tau_1^{-1}}\right) = L_1
\end{equation*}
(the last inclusion follows inductively from the hypothesis that $L_{\tau_i}L_{\tau_i^{-1}}=L_1$). 

Now fix $\gamma$ and $\delta\in \Gamma$. By local unitality we have $L_{\eta}L_1 = L_{\eta}$ for all $\eta$. Furthermore, by the above claim we have 
\begin{equation*}
  L_{\gamma\delta} = L_{\gamma\delta}L_1 = L_{\gamma\delta} L_{\delta^{-1}}L_{\delta}\subseteq L_{\gamma} L_{\delta}.
\end{equation*}
This is precisely what strong gradedness means, finishing the proof.
\end{proof}

\begin{proposition}\label{pr.dual-pb}
  Let $\Gamma$ be an arbitrary graph and $G\subseteq \bS^1$ a closed subgroup. Then, the gauge action of $G$ on $\bS^1$ is free if and only if the induced grading on the Leavitt path algebra $L(\Gamma)$ is strong.  
\end{proposition}
\begin{proof}
  Set $B=C^*(\Gamma)$ and $L=L(\Gamma)$ (the Leavitt path algebra, as above). We prove the two claimed implications separately.

  {\bf ($\Leftarrow$)} Suppose the action on $L$ induces a strong grading. Then in particular
  \begin{equation*}
    L^*_{\tau}L_{\tau} = L_{\tau^{-1}}L_{\tau}=L_1,\ \forall \tau\in \widehat{G}, 
  \end{equation*}
and hence $L_1$ is contained in $B^*_{\tau}B_{\tau}$. Since $L_1\subseteq B_1$ is dense, the conclusion follows. 

  {\bf ($\Rightarrow$)} We have to argue that $L_{\gamma}L_{\delta}=L_{\gamma+\delta}$ for arbitrary degrees $\gamma$ and $\delta$ in $\mathbb{Z}/k$ (including, possibly, $k=\infty$), assuming that the analytic version of this containment holds (involving closed linear spans, etc.).

  Despite the fact that $L$ is not in general unital, it is {\it locally} unital as a $\widehat{G}$-graded algebra in the sense of \Cref{def.loc-unit}. For this reason, \Cref{le.loc-unit-prel} ensures that it is enough to prove $L_d L^*_d=L_0$ for $d=\pm 1$. We do this for $d=1$, as the other case poses no additional difficulties. Moreover, in order to fix ideas, assume throughout the rest of the proof that $k=\infty$ (i.e. we are considering the full gauge action). 

Fix a projection $x=s_{\mu}s^*_{\mu}$ in $L_0$. Since words $s_{\mu}s^*_{\nu}$ with $|\nu|=|\mu|$ are dominated by $x$ on the left, it is enough to show that $x\in L_1 L^*_1$.  

By the freeness of the action on $B$ we have 
\begin{equation}\label{eq:lin-comb}
  \left\|\sum_{i,j=1}^t c_{i,j} x_i y^*_j -x\right\|<1,
\end{equation}
where $x_i=s_{\mu_i}s^*_{\nu_i}$ and $y_j=s_{\alpha_j}s^*_{\beta_j}$ for $|\mu_i|=|\nu_i|+1$ and similarly for the $\alpha$s and $\beta$s.

Multiplying the difference from (\ref{eq:lin-comb}) on the left and right by the projection $x$, we may as well assume that all $x_i y^*_j$ are contained in $xLx$.

Now consider the finite-dimensional $C^*$-algebra generated by $x$ and all $x_i y_j^*$ (for which $x$ serves as a unit). The latter all belong to $B_1 B^*_1$ and hence
\begin{equation*}
  M':=\mathrm{alg}\{x_i y^*_j\} \subset B_1 B^*_1. 
\end{equation*}
Now, \Cref{eq:lin-comb} implies that the subalgebra $M'\subseteq M$ contains an invertible element, and hence $1\in M'\subset B_1 B^*_1$. This is what we sought to prove.
\end{proof}

Via \Cref{pr.dual-pb}, the strong gradedness results of \cite[Theorem 4.2]{chr} can be paraphrased as follows.

\begin{theorem}\label{th.chr-th4.2}
  For a graph $\Gamma$ the full gauge action on $C^*(\Gamma)$ is free iff $\Gamma$
  \begin{itemize}
  \item has no sinks;
  \item is row-finite; 
  \item every infinite path has a finite sub-path $\alpha$ such that some path $\beta$ of length $|\alpha|+1$ shares the target of $\alpha$.    
  \end{itemize}
  \qedhere
\end{theorem}

\begin{remark}\label{re.condy}
  The third item in the statement is a rephrasing of condition (Y) of \cite[Definition 4.1]{chr}. 
\end{remark}

In particular, for graphs with finitely many vertices the third condition holds automatically (because every infinite path will eventually enter a cycle), so in that simplified case we have

\begin{corollary}\label{cor.chr-th4.2}
  For a graph $\Gamma$ with finitely many vertices the full gauge action on $C^*(\Gamma)$ is free if and only if $\Gamma$ is row-finite and has no sinks. 
  \qedhere
\end{corollary}

For finite actions, on the other hand, we have (see \cite[Proposition 4.3]{chr})

\begin{theorem}\label{th.chr-pr4.3}
  For a graph $\Gamma$ the gauge action by $\bZ/k$, $k\ge 2$ is free iff all sinks and infinite emitters are targets of length-$(k-1)$ paths.  
  \qedhere
\end{theorem}

A graph has no sinks if and only if the correspondence $E=E(\Gamma)$ of \Cref{subse.gr-cp} has faithful $\varphi:A\to B(E)$ (this is one of the remarks towards the end of \cite[$\S$3.4]{kat-init}). On the other hand, $\Gamma$ has no sinks precisely when $\varphi(A)\subseteq K(E)$ (e.g. \cite[p. 14]{tom-tlk}), which in the context of \Cref{subse.full} means that $1\in K(E)$, i.e. $E$ is FG.

The discussion thus far matches up two of the conditions of \Cref{th.cp-iff} (faithfulness and finite generation) respectively to two of the conditions listed in \Cref{th.chr-th4.2} (no sinks and row-finiteness). In order to recover \Cref{cor.chr-th4.2} from \Cref{th.cp-iff} in the finite-$\Gamma^0$ case, it remains to argue that the third condition of \Cref{th.cp-iff} is satisfied automatically for such graphs.

\begin{lemma}\label{le.art-gr}
The correspondence $E(\Gamma)$ associated to a graph $\Gamma$ with finitely many vertices is artinian in the sense of \Cref{def.artin}. 
\end{lemma}
\begin{proof}
  This is immediate: the lattice of ideals of $C_0(\Gamma^0)$ is isomorphic to the lattice of subsets of $\Gamma^0$, and when the latter is finite the lattice is too.
\end{proof}

\begin{remark}\label{re.art-gr}
  This was not necessary in the proof of \Cref{le.art-gr}, but it is perhaps worth noting that for arbitrary $\Gamma$ the ideal $I_n$ from \Cref{def.artin} is $C_0(S_n)\subset C_0(\Gamma^0)$, where $S_n\subseteq \Gamma^0$ is the set of vertices that are targets for paths of length $\ge n$.  
\end{remark}

In conclusion, \Cref{cor.chr-th4.2} is a consequence of \Cref{th.cp-iff}. We next turn to the finite-action case. Here, it turns out that the necessary and sufficient conditions for freeness in \Cref{th.fin} are equivalent without finiteness assumptions.

\begin{proposition}\label{pr.fin-gr}
  Let $\Gamma$ be an arbitrary graph, $k\ge 2$ a positive integer and $(A,E,\varphi)$ the correspondence associated to it. The latter satisfies the condition \Cref{eq:k-1} of \Cref{th.fin} if and only if all sinks and infinite emitters receive paths of length $k-1$
\end{proposition}
\begin{proof}
  For correspondences associated to graphs we have
  \begin{equation*}
    \varphi^{-1}(K(E)) = C_0(\text{set of finite emitters}),\ \ker(\varphi)=C_0(\text{set of sinks})
  \end{equation*}
  (see for instance \cite[p.14]{tom-tlk}). 
  It then follows that the spectrum of $J_E$ is the set of vertices in $\Gamma^0$ that are both finite emitters and {\it not} sinks.

  Since on the other hand the spectrum of $I_{k-1}$ is the set of vertices receiving paths of length $\ge k-1$ (\Cref{re.art-gr}), condition \Cref{eq:k-1} means precisely that all vertices that are either sinks or infinite emitters receive such a path.
\end{proof}

Once more, this means that \Cref{th.chr-pr4.3} follows (for graphs with finitely many vertices) from \Cref{th.fin}.


\section*{Acknowledgments}

This work is partially supported by NSF grant DMS-1801011, and part of the project supported by the grant H2020-MSCA-RISE-2015-691246-QUANTUM DYNAMICS, grant number 3542/H2020/2016/2.

I am grateful for numerous enlightening conversations with S\o ren Eilers, Piotr Hajac, Benjamin Passer and Mariusz Tobolski on the contents of this paper and related topics that are the subject of future joint work.



\begin{thebibliography}{10}

\bibitem{freeness}
Paul~F. Baum, Kenny De~Commer, and Piotr~M. Hajac.
\newblock Free actions of compact quantum groups on unital {$C^*$}-algebras.
\newblock {\em Doc. Math.}, 22:825--849, 2017.

\bibitem{ell}
David~Alexandre Ellwood.
\newblock A new characterisation of principal actions.
\newblock {\em J. Funct. Anal.}, 173(1):49--60, 2000.

\bibitem{flr}
Neal~J. Fowler, Marcelo Laca, and Iain Raeburn.
\newblock The {$C^*$}-algebras of infinite graphs.
\newblock {\em Proc. Amer. Math. Soc.}, 128(8):2319--2327, 2000.

\bibitem{kat-init}
Takeshi Katsura.
\newblock A construction of {$C^*$}-algebras from {$C^*$}-correspondences.
\newblock In {\em Advances in quantum dynamics ({S}outh {H}adley, {MA}, 2002)},
  volume 335 of {\em Contemp. Math.}, pages 173--182. Amer. Math. Soc.,
  Providence, RI, 2003.

\bibitem{kat-alg}
Takeshi Katsura.
\newblock On {$C^*$}-algebras associated with {$C^*$}-correspondences.
\newblock {\em J. Funct. Anal.}, 217(2):366--401, 2004.

\bibitem{kat-ideal}
Takeshi Katsura.
\newblock Ideal structure of {$C^*$}-algebras associated with
  {$C^*$}-correspondences.
\newblock {\em Pacific J. Math.}, 230(1):107--145, 2007.

\bibitem{kpr}
Alex Kumjian, David Pask, and Iain Raeburn.
\newblock Cuntz-{K}rieger algebras of directed graphs.
\newblock {\em Pacific J. Math.}, 184(1):161--174, 1998.

\bibitem{chr}
L.~{Orloff Clark}, R.~{Hazrat}, and S.~W. {Rigby}.
\newblock {Strongly graded groupoids and strongly graded Steinberg algebras}.
\newblock {\em ArXiv e-prints}, November 2017.

\bibitem{phi-equivk}
N.~Christopher Phillips.
\newblock {\em Equivariant {$K$}-theory and freeness of group actions on
  {$C^*$}-algebras}, volume 1274 of {\em Lecture Notes in Mathematics}.
\newblock Springer-Verlag, Berlin, 1987.

\bibitem{phi-free}
N.~Christopher Phillips.
\newblock Freeness of actions of finite groups on {$C^*$}-algebras.
\newblock In {\em Operator structures and dynamical systems}, volume 503 of
  {\em Contemp. Math.}, pages 217--257. Amer. Math. Soc., Providence, RI, 2009.

\bibitem{pim}
Michael~V. Pimsner.
\newblock A class of {$C^*$}-algebras generalizing both {C}untz-{K}rieger
  algebras and crossed products by {${\bf Z}$}.
\newblock In {\em Free probability theory ({W}aterloo, {ON}, 1995)}, volume~12
  of {\em Fields Inst. Commun.}, pages 189--212. Amer. Math. Soc., Providence,
  RI, 1997.

\bibitem{sch}
Hans-J\"urgen Schneider.
\newblock Principal homogeneous spaces for arbitrary {H}opf algebras.
\newblock {\em Israel J. Math.}, 72(1-2):167--195, 1990.
\newblock Hopf algebras.

\bibitem{seg}
Graeme Segal.
\newblock Equivariant {$K$}-theory.
\newblock {\em Inst. Hautes \'Etudes Sci. Publ. Math.}, (34):129--151, 1968.

\bibitem{s-w03}
Wojciech Szyma\'nski.
\newblock Quantum lens spaces and principal actions on graph {$C^*$}-algebras.
\newblock In {\em Noncommutative geometry and quantum groups ({W}arsaw, 2001)},
  volume~61 of {\em Banach Center Publ.}, pages 299--304. Polish Acad. Sci.
  Inst. Math., Warsaw, 2003.

\bibitem{tom-tlk}
Mark Tomforde.
\newblock Graph {$C^*$}-algebras as {C}untz-{P}imsner algebras.
\newblock available on author's website; accessed 2018-05-18.

\bibitem{tu-eq}
Loring~W. Tu.
\newblock What is {$\dots$} equivariant cohomology?
\newblock {\em Notices Amer. Math. Soc.}, 58(3):423--426, 2011.

\bibitem{wo}
N.~E. Wegge-Olsen.
\newblock {\em {$K$}-theory and {$C^*$}-algebras}.
\newblock Oxford Science Publications. The Clarendon Press, Oxford University
  Press, New York, 1993.
\newblock A friendly approach.

\end{thebibliography}
\bibliographystyle{plain}

\def\polhk#1{\setbox0=\hbox{#1}{\ooalign{\hidewidth
  \lower1.5ex\hbox{`}\hidewidth\crcr\unhbox0}}}
  \def\polhk#1{\setbox0=\hbox{#1}{\ooalign{\hidewidth
  \lower1.5ex\hbox{`}\hidewidth\crcr\unhbox0}}}
  \def\polhk#1{\setbox0=\hbox{#1}{\ooalign{\hidewidth
  \lower1.5ex\hbox{`}\hidewidth\crcr\unhbox0}}}
  \def\polhk#1{\setbox0=\hbox{#1}{\ooalign{\hidewidth
  \lower1.5ex\hbox{`}\hidewidth\crcr\unhbox0}}}
  \def\polhk#1{\setbox0=\hbox{#1}{\ooalign{\hidewidth
  \lower1.5ex\hbox{`}\hidewidth\crcr\unhbox0}}}

\end{document}